\documentclass[12pt, notitlepage]{amsart}
\usepackage{latexsym, amsfonts, amsmath, amssymb, amsthm, cite}

\newtheorem{lemma}{Lemma}

\newtheorem{theorem}{Theorem}

\newtheorem{corollary}{Corollary}
\newtheorem*{conjecture}{Schinzel's Hypothesis H}

\usepackage{graphicx}
\usepackage{mathrsfs}

\begin{document}

\title{Common divisors of totients of polynomial sequences} 
\author{J. Br{\"u}dern} 
\address{Mathematisches Institut\\ Universit\"at G\"ottingen\\ Bunsenstrasse 3--5, D-37073 G\"ottingen}
\email{joerg.bruedern@mathematik.uni-goettingen.de} 
 \author{K. Soundararajan} 
\address{Department of Mathematics \\ Stanford University \\
450 Serra Mall, Bldg. 380\\ Stanford, CA 94305-2125}
\email{ksound@stanford.edu}
  
 \maketitle
  
\section{Introduction}  

\noindent Let $f\in{\mathbb Z}[x]$ be a primitive  polynomial of degree $k$ (that is, the coefficients of $f$ have gcd equal to $1$).   We are interested in 
\begin{equation} 
\label{1} 
{\mathcal G}(f) = \text{gcd} \{ \phi(f(n)): \  \ n\in {\Bbb N}\}. 
\end{equation}  
In particular we are motivated by the question of Venkataramana \cite{V, VMo} that ${\mathcal G}(f)$ is bounded 
by a number ${\mathcal G}_k$ depending only on the degree $k$ of the polynomial $f$. He handled  the 
case of linear polynomials and found that ${\mathcal G}(f)\mid 4$ holds for all $f(n)=an+b$ with $(a,b)=1$. The polynomials $n$, $2n+1$ and $16n+5$ show that  ${\mathcal G}(f)$ takes all three admissible values 1, 2 and 4. Results of this type have been applied to the congruence
subgroup problem, and as Venkataramana points out, in this context Serre
\cite{S} had obtained {\em inter alia} that ${\mathcal G}(f)$ is a divisor of 8 for all linear $f$.

In this note we are concerned with ${\mathcal G}(f)$ for polynomials of higher degree.  In brief, we are able to establish the existence of an admissible value for ${\mathcal G}_2$, and we also give a bound for  ${\mathcal G}(f)$ when the polynomial $f$ splits completely 
into linear factors. 
  Assuming the Schinzel conjectures on prime values taken by polynomials we are able to describe the factorisation of  ${\mathcal G}(f)$ quite precisely, for all polynomials, and thereby establish the existence of ${\mathcal G}_k$. Examples will demonstrate that the conditional results are optimal in some cases.  
We now describe our results more precisely.
Let us recall first Schinzel's hypothesis.

\begin{conjecture} Let $F_1$, $\ldots$, $F_r$ be irreducible polynomials 
with integer coefficients, and positive leading coefficients.  Suppose that the product $F_1 \cdots F_r$ is 
not divisible by any fixed prime.  Then there are infinitely many natural numbers $n$ such that $F_j(n)$ is 
prime for each $1\le j\le r$.  
\end{conjecture}  

\begin{theorem}  Assume Schinzel's hypothesis.  Let $f=f_1^{a_1} \cdots f_s^{a_s}$ be a primitive 
polynomial with integer coefficients, with the $f_j$ being distinct irreducibles of degree $k_j$.   Let $r_j$ be the maximal integer 
such that $K_{f_j}$ (a field obtained by adjoining a root of $f_j$ to ${\Bbb Q}$) contains the $r_j$-th roots of 
unity.  Then $\phi(r_j)|k_j$, and ${\mathcal G}(f)$ divides $\phi(k!) r_1^2 \cdots r_s^2$.  
\end{theorem}  

In the case of a linear polynomial $f$ the proof of Theorem 1 will call upon Schinzel's hypothesis only for linear polynomials, whence that case depends on Dirichlet's theorem, 
and we recover Venkataramana's result unconditionally.   

{\bf Example 1.}  Suppose $f(n) =\prod_{j=1}^{k} (a_j n+b_j)$ is the product of $k$ primitive linear polynomials.  Here Theorem 1 gives ${\mathcal G}(f) \mid 4^k \phi(k!) $.   When 
$k=2$, the polynomial $f(n)=(16n+5)(16n+13)$ has  ${\mathcal G}(f)= 16$, matching the bound of Theorem 1.   More generally, if we consider $f(n) = (n+1) \cdots (n+k)$, then $k!$ divides $f(n)$ for all $n$ and so $\phi(k!)$ divides ${\mathcal G}(f)$.   So the result in Theorem 1 is tight except perhaps for the power of $2$ dividing ${\mathcal G}(f)$.

 Since quadratic fields have only $2$, $4$, or $6$ roots of unity, if $f$ is a primitive irreducible polynomial of degree $2$ then by Theorem 1 the possible 
values of ${\mathcal G}(f)$ must be divisors of $36$ or $16$ (assuming Schinzel's hypothesis).   
We now give examples to show that this cannot be sharpened.  

{\bf Example 2.}   Consider the polynomial $f(n)=16n^2+1$, which takes values $\equiv 1 \bmod 16$. 
The prime divisors of $f(n)$ are congruent to $1 \bmod 4$. Hence, 
if $f(n)$ has at least two distinct prime factors $p_1,p_2$, say, then $16|(p_1-1)(p_2-1)$, and therefore, 
$16\mid \phi(f(n))$.
It remains to consider the case where $f(n)$ is a power of a prime $p$.  Since $f(n)=(4n)^2 +1$ can never be a 
perfect square for $n\ge 1$, we may restrict attention to $f(n)= p^{\ell}$ with $\ell$ odd.  But then $p$ must be $1 \bmod 16$, and once again $16$ 
divides $\phi(f(n))$.    This proves that ${\mathcal G}(f) =16$, with the convention that the natural numbers start at $1$.  
If the natural numbers start at $0$, simply consider $f(n+1)$.   More generally, by shifting a polynomial by a large integer, we may 
discard any finite set of undesired values in understanding ${\mathcal G}$.  The reader may wish to construct irreducible quadratic polynomials where  ${\mathcal G}(f)$ is a given proper divisor of $16$.   


{\bf Example 3.}  Start with $f_0(n) = n^2+n+1$, and consider $f(n)= f_0(72n)$.  The values of $f$ are all $\equiv 1 \bmod 72$, and any prime factor of $f(n)$ 
must be $\equiv 1 \bmod 6$.  Thus if $f(n)$ is divisible by two distinct primes then $\phi(f(n))$ will be a multiple of $36$.   If $f(n)$ is prime, then $\phi(f(n))$ 
will be a multiple of $72$.   It remains to consider the case $f(n)= p^{\ell}$ for $\ell \ge 2$.  If $3\nmid \ell$ and $4 \nmid \ell$ then $p^\ell \equiv 1 \bmod 72$ implies that $p\equiv 1 \bmod 36$, and once again $36$ divides $\phi(f(n))$.   The last remaining possibilities entail that $f(n)$ is either a cube or a fourth power.  
Since these correspond to integer points on two curves of positive genus, there are only finitely many such $n$ (which we could certainly determine in this example).  By translating the polynomial $f$ if necessary, we can avoid these finitely many examples, and arrive at a cubic polynomial $\widetilde{f}$ 
with $36 | {\mathcal G}(\widetilde{f})$.  Similar examples can be constructed starting with other cyclotomic polynomials; for instance starting with $n^4+n^3+n^2+n+1$ we can find a quartic polynomial $f$ with $25 | {\mathcal G}(f)$.

In the above examples, we were led to consider when a polynomial with integer coefficients and degree at least $2$ takes pure power 
values.  We note, in passing, the work of Schinzel and Tijdeman \cite{ST} which ensures that if the polynomial has at least three simple zeros 
then there are only finitely many such pure power values.

\medskip

Suppose $f$ splits completely into linear factors.  Then, as noted above,  our conditional Theorem 1 tells us that ${\mathcal G}(f)$ 
is a divisor of $2^{2k} \phi(k!)$.   In this situation, we can give an unconditional bound for the possible values of ${\mathcal G}(f)$. 

\begin{theorem}  Suppose $f$ is a primitive polynomial of degree $k$ splitting completely into linear factors.  Then ${\mathcal G}(f)$ 
is not divisible by any prime larger than $2k+1$.  Moreover, for every prime $\ell \le 2k+1$ there exists a constant $C(k,\ell)$ such that the power of $\ell$ dividing ${\mathcal G}(f)$ is at most $C(k,\ell)$.  
\end{theorem} 

We are also able to show unconditionally that  ${\mathcal G}_2$ is finite. In view of the preceding theorem, it is enough to 
consider primitive irreducible quadratic polynomials.

\begin{theorem}  There is a number $G$ with the property that for 
all primitive and irreducible quadratic polynomials $f$ with positive leading coefficient one has ${\mathcal G}(f)
\le G$.
\end{theorem}

The proofs of the unconditional results depend on the fundamental lemma
in sieve theory, and the switching principle. When discussing irreducible quadratic polynomials we will have to rely also on quantitative estimates concerning the equidistribution of the roots of quadratic polynomials, a subject initiated by Hooley \cite{H}. We require bounds for averages of Weyl sums associated with these roots, twisted with a Dirichlet character. Such bounds follow from the work of Toth \cite{T}. His work in turn is inspired by important  contribution by Duke, Friedlander and Iwaniec \cite{DFI}. Along the way, we prove an auxiliary result that is of some interest in its own right. 

\begin{theorem}  Let $f$ be an irreducible quadratic polynomial with no fixed prime factor.   There 
exist absolute constants $\delta$ and $h_0$ with the following property.   If $h>h_0$ then 
there are infinitely many $n$ such that $f(n)$ is divisible by no prime below $n^{\delta}$, and by no prime $p\equiv 1\pmod h$. 
\end{theorem}  

This result provides an affirmative answer to a question that Frank Calegari \cite{C} put forward 
in his blog: are there infinitely many values of $n^2+1$ that are not divisible by primes
$\equiv 1 \bmod 2^m$, at least when $m$ is a fixed large integer? 



\smallskip 

\noindent{\bf Acknowledgements.}  The second author is partially supported by a grant from the National Science Foundation and a 
Simons Investigator award from the Simons Foundation.  This paper was begun while the second author was a Gauss Visiting Professor 
at G{\" o}ttingen, supported by the Akademie der Wissenschaften zu G{\" o}ttingen, and completed while he was a Senior Fellow 
at the Institute for Theoretical Studies, ETH Z{\" u}rich.  He thanks both institutions for their warm and generous hospitality.  

\section{Preliminary reductions}  

Although the coefficients of $f$ have no common factor, it may still be that the values $f(n)$ for $n\in {\Bbb N}$ 
have a common factor.  Our first lemma allows us to get rid of this common factor, and restrict attentions to polynomials for which the values have no non-trivial common factor.  

\begin{lemma} 
\label{lem1}  Let $d$ denote the greatest common factor of $f(n)$ for all $n\in {\Bbb N}$.  Then 
$d$ is a divisor of $k!$.  Moreover, with $D= d\prod_{p\le k} p$ we may find a progression $a+ Dn$ such 
that $F(n) = f(a+nD)/d$ is a polynomial with integer coefficients and with $F(n)$ being coprime to $\prod_{p\le k} p$ 
for all $n$.   Finally, 
${\mathcal G}(f)$ is a divisor of $\phi(d) {\mathcal G}(F)$.  
\end{lemma} 
\begin{proof}  Write the polynomial $f$ in the basis of binomial coefficients: $f(x) = b_0 \binom{x}{0} + b_1 \binom{x}{1} + \ldots + b_k \binom{x}{k}$.  By considering the values $x=0$, $1$, $\ldots$, $k$ we see that the greatest common factor of all the $f(n)$ is simply the greatest common factor of these coefficients $b_0$, $\ldots$, $b_k$.  Since the denominators appearing in the binomial coefficients all divide $k!$, clearly the common factor $d$ must be a divisor of $k!$.

Suppose $p\le k$ and  $p^{\alpha} \Vert d$.  Then there must exist a residue class $a_p \bmod{p^{\alpha+1}}$ with $p^{\alpha} \Vert f(n)$ for all $n\equiv a_p \bmod{p^{\alpha+1}}$.  Thus by the chinese remainder theorem we may find a progression $a+Dn$ with $F(n) = f(a+nD)/d$ being a polynomial with integer coefficients and  all values of $F$ 
being coprime to $\prod_{p\le k} p$.  This proves our second assertion, and the third follows at once.  
\end{proof}

\begin{lemma} 
\label{lem2}  Let $f$ be an irreducible polynomial in ${\Bbb Z}[x]$, and let $K_f$ be a 
field obtained by adjoining some root of $f$ to ${\Bbb Q}$.  Given a natural number $m$, the following two conditions are equivalent: 

(i) $K_f$ contains the $m$-th roots of unity. 

(ii) All but finitely many of the primes $p$ that divide the values of $f$ satisfy $p \equiv 1\pmod{m}$.  
\end{lemma} 
\begin{proof}  If $K_f$ contains the $m$-th roots of unity, then an ideal of norm $p$ in $K_f$ must lie above a 
prime of norm $p$ in ${\Bbb Q}(e^{2\pi i/m})$, and therefore $p$ must be $1 \pmod m$.  Thus (i) implies (ii).  

That (ii) implies (i) follows upon applying the Chebotarev density theorem to the extension $K_f(e^{2\pi i/m})$ of $K_f$, obtained by adjoining (if necessary) the $m$-th roots of unity to $K_f$.    The assumption (ii) means that if there is a 
prime of norm $p$ in $K_f$ then $p \equiv 1 \pmod m$, but then the Frobenius at any such prime in $K_f$ acts trivially on the $m$-th roots of unity.  Thus for almost all primes of degree $1$ in $K_f$ the Frobenius action on 
$K_f(e^{2\pi i/m})$ is the identity, which means that the degree $[K_f(e^{2\pi i/m}),K_f]$ must be $1$.  
\end{proof}

  \begin{lemma} 
  \label{lem3}   Assume Schinzel's Hypothesis.   Let $f$ be an irreducible polynomial of degree $k$, and such that the values of $f$ are coprime to all 
  the primes at most $k$.   Let $\ell$ be a prime, and suppose that $K_f$ contains the $\ell^{\alpha}$-th roots of unity, but not the $\ell^{\alpha+1}$-th roots.  Let $\beta$ be the largest integer such that $\ell^{\beta}$ divides $f(n)-1$ for all $n$.  Then the largest power of $\ell$ that divides ${\mathcal G}(f)$ is at most $\min (\beta, 2\alpha)$.  
      \end{lemma} 
 \begin{proof}   Since $\beta$ is the largest power of $\ell$ dividing $f(n)-1$, we may find a progression 
 $a \bmod{\ell^{\beta+1}}$ with $f(a) \not \equiv 1 \bmod{\ell^{\beta+1}}$.   Restrict to this progression.  Since the 
 values of $f$ are coprime to the primes below $k$, the polynomial $f(a+n \ell^{\beta+1})$ does not have a common prime factor.  Therefore by Schinzel's hypothesis, we may find values of $n$ with $f(a+n\ell^{\beta+1})=p$ being 
 prime, and the largest power of $\ell$ dividing $p-1$ is $\beta$.  This shows that the largest power of $\ell$ dividing ${\mathcal G}(f)$ is at most $\beta$, which proves the lemma when $\beta \le 2\alpha$.  
   
 
 Now suppose that $\beta \ge 2\alpha+1$.   By Lemma 2 we know that all but finitely many of the primes dividing $f(n)$ are $1\bmod{\ell^{\alpha}}$ and also that there are infinitely many primes dividing $f(n)$ that are $\not\equiv 1\bmod{\ell^{\alpha+1}}$.   Pick a large prime $q \not\equiv 1 \bmod{\ell^{\alpha+1}}$ and a residue class $a \pmod q$ such that $q$ divides $f(a)$ but $q^2 \nmid f(a)$.     Then the polynomial  $f(a+xq)/q$ is irreducible, has no prime common factor (since the values of $f(n)$ have no prime factor at most $k$), and therefore takes prime values infinitely often.  Let $p$ be one such prime value.  Since the values of $f$ are $1\bmod {\ell^{\alpha+1}}$ and $q\not \equiv 1 \bmod{\ell^{\alpha+1}}$ we know that $p$ also is not $1 \bmod{\ell^{\alpha+1}}$.  Therefore the largest power of $\ell$ dividing ${\mathcal G}(f)$ must also divide $(p-1)(q-1)$, which completes our proof. 
   \end{proof}

  \begin{corollary}  Assume Schinzel's Hypothesis.  
   Let $f$ be an irreducible polynomial of degree $k$, and such that all the values $f(n)$ are not divisible by any prime at most $k$.   Let $r$ be the maximal integer such that $K_f$ contains the $r$-th roots of unity.  
  Then $\phi(r)$ divides $k$, and  ${\mathcal G}(f)$ divides $r^2$.  
  \end{corollary}     
  
  Now we want to proceed to the general case of a polynomial of degree $k$, not necessarily irreducible.   We begin with a simple observation. 
 
 \begin{lemma} 
 \label{lem4}  If $f$ and $g$ are two coprime polynomials in ${\Bbb Z}[x]$ then for all large primes $q$ either 
 at most one of $f(n)$ or $g(n)$ can be divisible by $q$.  
 \end{lemma}  
 \begin{proof}  By the Euclidean algorithm we may find polynomials $u$ and $v$ with integer coefficients and 
 a non-zero integer $c$ such that $f(x)u(x) + g(x) v(x) = c$.  Thus if $q\nmid c$, then $q$ can divide at most one of 
 $f(n)$ or $g(n)$.  
 \end{proof} 
 
 Now we are ready for the general form of Lemma \ref{lem3}. 
 
 \begin{lemma} 
 \label{lem5} Assume Schinzel's Hypothesis.  Let $f$ be a polynomial of 
  degree $k$ whose values contain no prime factor at most $k$.   Suppose $f$ factors as $f_1^{a_1} \cdots f_s^{a_s}$ 
  where the $f_j$ are distinct irreducible polynomials and $a_j \ge 1$. Let $\ell$ be a prime, and suppose $\ell^{\alpha_j}$ is the largest power of $\ell$ such that $K_{f_j}$ contains the $\ell^{\alpha_j}$-th roots of unity.   Then the largest power of $\ell$ dividing ${\mathcal G}(f)$ is at most $2(\alpha_1 +\ldots +\alpha_s)$. 
 \end{lemma}  
 
 \begin{proof}  Choose a value $a$ such that $f_j(a) \neq 1$ for all $j$.  Then $\ell^{\beta} \Vert \prod_{j} (f_j(a)-1)$ 
for some non-negative integer $\beta$.  Below we shall restrict ourselves to the progression $n\equiv a\bmod{\ell^{\beta+1}}$.   Clearly on this 
progression we have $\ell^{\beta_j} \Vert (f_j(n)-1)$ for all $j$, and some non-negative integers $\beta_j$.  
 
 If $\beta_j \le \alpha_j$ put $q_j=1$ and $a_j=1$.  If $\beta_j >\alpha_j$ then select a large prime $q_j \not\equiv 1\bmod{\ell^{\alpha_j+1}}$, and a residue class $a_j \bmod{q_j}$ such that $q_j \Vert f_j(a_j)$ --- this is possible in view of Lemma 2, and $q_j$ is chosen large enough so that it does not divide the discriminant of $f_j$.  Our choice of $q_j$ will be such that no two $q_j >1$ are equal.  Now consider $n$ lying in the progressions $a \bmod{\ell^{\beta+1}}$ and $a_j\bmod{q_j}$ for all $j$.   We apply Schinzel's Hypothesis to the polynomials $f_j(n)/q_j$ for $n$ in this progression.  The primes $q_j$ are chosen large enough so that they do not divide the resolvent of any two polynomials $f_j$.  Then there can be no fixed prime common to all the $f_j$, and  Schinzel's Hypothesis is applicable.  
 
 What is the power of $\ell$ dividing $\phi(\prod_{j} f_j(n)^{a_j})$?  By construction this is the power of $\ell$ in $\prod_{j} \phi(q_j) (f_j(n)/q_j-1)$.  The terms with $q_j=1$ contribute a power of $\beta_j\le \alpha_j$, while the terms with $q_j>1$ contribute a power of $2\alpha_j$.  This completes our proof.
 \end{proof}

 \section{Proof of Theorem 1}  
 
 With the results from Section 2 in hand, we can finish the proof of  Theorem 1 in a few sentences.  Given a primitive polynomial $f=f_1^{a_1} \cdots f_s^{a_s}$, by passing to a progression (as in Lemma 1) we may find a polynomial $F = F_1^{a_1} \cdots F_s^{a_s}$ 
 with $F(n)$ coprime to all the primes below $k$ and with ${\mathcal G}(f)$ being a divisor of $\phi(k!) {\mathcal G}(F)$.   Further, the fields obtained by adjoining a root of $f_j$ to ${\Bbb Q}$ are the same as the fields obtained by adjoining a root  of $F_j$ to ${\Bbb Q}$.    Thus, appealing to Lemma 5, we find that ${\mathcal G}(F)$ is a divisor of $r_1^2 \cdots r_s^2$.   This completes our proof.

 \section{Polynomials that split completely: Proof of Theorem 2} 
 
We turn to the proof of Theorem 2. If $f$ is primitive and splits completely into linear factors, then
$$ 
f(n) = \prod_{j=1}^s (c_j n + d_j)^{a_j} 
$$
with the $c_j$ non-zero, $(c_j,d_j)=1$ for all $j$, and the rationals $d_j/c_j$ distinct. We may suppose that $f$ has positive leading coefficient, and then we can arrange matters such that the $c_j$ are all positive. We require the ``square-free kernel" of $f$, given by
$$ 
g(n) = \prod_{j=1}^s (c_j n + d_j). 
$$
Further we suppose that $f$ is not divisible by primes at most $2k+1$, with $k$ the degree of $f$. An obvious variant of Lemma 1 allows us to do so.  
 
First let us show that no prime $\ell > 2k+1$ divides ${\mathcal G}(f)$. 
By passing to a progression $\bmod{\ell}$ we 
may suppose that $\ell \nmid f(n)$ for all $n$.   Now $\ell$ can divide $\phi(f(n))$ if and only if $f(n)$ is divisible by 
some prime $p\equiv 1\bmod \ell$.  Consider the sifting problem of finding $n$ such that $c_jn\not \equiv -d_j \bmod {p}$ for all $j$, and all primes $p\equiv 1\bmod \ell$.   This is a sieve of dimension $s/(\ell-1) <1/2$,
and the sequence to be sifted, with $n\le x$, has level of distribution $x^{1-\epsilon}$, for any $\epsilon >0$.
Sieve theory in dimension below half therefore shows that there are (many) values of $n$ with the desired property  (see, for example, \cite[Theorem 11.21]{FI}).

\smallskip 

Now consider a prime $\ell \le 2k+1$, where we wish to show that the power of $\ell$ dividing ${\mathcal G}(f)$ is bounded.   Let $z$ be a large parameter, and let $P(z)$ denote the product of all primes below $z$. Further, let $A$ be a large natural number.  We seek a lower bound for
$$ 
S(x,z) = \sum_{\substack{n\le x \\ (g(n), P(z)) =1 }} \Big(1 - \sum_{\substack{ z <p \le x \\ p\equiv 1 \bmod{\ell^A}} } \sum_{\substack { j \\ p| c_j n+d_j }} 1 \Big). 
$$ 
Here, we begin by estimating the positive contributions. Let $\varrho(p)$ denote the number of incongruent solutions
to $g(n) \equiv 0 \mod p$. For large primes $p$ we then have $\varrho(p)=s$, and hence the product
\begin{equation}\label{3.1} {\mathfrak S} = \prod_p \Big(1-\frac{\varrho(p)}{p}\Big) \Big(1-\frac1p\Big)^{-s} \end{equation}
converges to a non-zero number. We may now apply the 
fundamental lemma of sieve theory in dimension $s$. This tells us that there is a positive real number $C$
such that for all $x \ge z^{9s}$, one has  
$$ 
\sum_{\substack{n\le x \\ (g(n), P(z)) =1 }} 1 \ge C {\frak S} \frac{x}{(\log z)^{s}}, 
$$ 
see for example \cite[Theorem 11.22]{FI}.
Note here that $\mathfrak S$ depends on $g$ but $C$ does not.

Now we turn to the contribution of the negative terms in the sum defining $S$. 
By reasons of symmetry it is enough to think of the case $p|c_1 n+d_1$,  say. 

Consider first the terms with $z < p \le x/z^{9s}$.  In this case, we apply an upper bound sieve,
for example again \cite[Theorem 11.22]{FI}. Then, with $\mathfrak S$ as above, we find that
contribution is bounded above by
$$ 
C' {\mathfrak S} \sum_{\substack {z < p \le x/z^{9s} \\ p\equiv 1\bmod{\ell^A}}} \frac{x}{p (\log z)^s} $$
where again $C'$ is a suitable positive constant that does not depend on $g$. We choose $x=z^{30s}$. Then, since $z$ is large, the above does not exceed
$$ 
\le 2C' {\mathfrak S} \frac{x}{(\log z)^s} \frac{1}{\ell^A} \log \frac{\log x/z^{9s}}{\log z} = 2C'
{\mathfrak S} \frac{x}{(\log z)^s} \frac{\log (21s)}{\ell^A}. 
$$ 

Now consider the contribution of larger values of $p$.  Here we employ the switching principle:  write 
$c_1 n + d_1 = r p$, and then sum over $r$ instead.   We must have $r\ll z^{9s}$ with $r$ composed only 
of prime factors above $z$, and moreover we must have $n$ in a particular residue class $\bmod{r\ell^A}$ (since 
$r|(c_1 n+d_1)$ and we must have $(c_1 n + d_1)/r \equiv 1\bmod{\ell^A}$, choosing to forget that it must also be prime).   Once more applying the upper bound sieve, the desired contribution is 
$$ 
\le C'{\mathfrak S} \sum_{\substack{ r\ll z^{9s} \\ (r,P(z))=1}} \frac{x}{(r\ell^A) (\log z)^s} \le C''{\mathfrak S} \frac{x}{(\log z)^s} 
\frac{\log (10s)}{\ell^A} 
$$ 
wheere now $C''$ is a suitable constant with $C''\ge C'$.

Combining the two upper bounds with the lower bound, we infer that (recall $x=z^{30s}$)
$$ 
S(z^{30s},z) \ge {\mathfrak S} \frac{x}{(\log z)^{s}} \Big(  C - C'' \frac{3s \log (21s)}{\ell^A} \Big), 
$$ 
We choose $A$ so large that $S(z^{30s},z) \ge\frac12 C {\mathfrak S} \frac{x}{(\log z)^{s}}.$
Thus we have produced $n$ for which $g(n)$ has at most $30s^2$ prime factors, and 
none of these prime factors can be $1\bmod{\ell^A}$. Then $f(n)$ will have at most $30k^2$ such prime factors.  Therefore the exponent of $\ell$ dividing ${\mathcal G}(f)$ 
may be bounded in terms of $A$ and $k$, as claimed.   

\section{Irreducible quadratic polynomials} 
Now familiar arguments show that Theorem 3 follows from Theorem 4. 
Thus it remains to establish the latter, and this is our
main task in this section.    The basic strategy is similar to that applied in the previous section.

Let $f(x) = ax^2 + bx+c \in{\mathbb Z}[x]$ be a primitive irreducible quadratic polynomial with positive leading coefficient and no fixed prime divisor.  Let $D = b^2 -4ac$ denote the discriminant of $f$, and put $H= 2a |D| h$.     
We fix a progression $\nu \bmod H$ such that $(f(\nu), H) =1$, and assume that $1\le \nu \le H$.   
Let $x$ be large, and put $z= x^{\delta}$ for a suitably small $\delta >0$.   Put $P^{\dagger} = \prod_{p\le z, p\nmid H} p$.  
We wish to bound from below 
\begin{equation} 
\label{4.1} 
S = \sum_{\substack{x\le n\le 2x \\ n \equiv \nu \bmod H \\ (f(n),P^{\dagger}) = 1 } } \Big(1 - \sum_{\substack{p\equiv 1 \bmod{h} \\ z\le p \le f(2x)\\ p\mid f(n)} } 1 \Big).  
\end{equation} 

We start with the positive term in $S$.   An application of the fundamental lemma from sieve theory, for example in the form of \cite[Thm. 6.12]{FI}, shows that 
$$ 
\sum_{\substack{ n\le x \\ n \equiv \nu \bmod H \\ (f(n),P^{\dagger}) = 1 } } 1 \ge \frac{1}{2}  \frac{x}{H} \prod_{\substack{ p\le z \\ p\nmid H}} \Big(1-\frac{\varrho(p)}{p} \Big),
$$ 
where ${\varrho}(p)$ denotes the number of solutions to the congruence $f(x) \equiv 0 \bmod p$.  
For a prime $p \nmid H$ (and so in particular $p\nmid 2aD$), it is easy to verify that ${\varrho}(p) =1 +  (\frac{D}{p})$ (where $(\frac{D}{\cdot})$ denotes the Kronecker--Legendre symbol, which is a Dirichlet character $\bmod {\ |D|}$).   On average $\varrho({p}) = 1$, and so for large $z$ we have 
$$ 
\prod_{\substack{p\le z \\ p\nmid H}} \Big(1 -\frac{\varrho(p)}{p}\Big)  \sim \frac{e^{-\gamma}}{\log z} \frac{H}{\phi(H)}  \prod_{\substack{p\le z \\ p\nmid H}} \Big( 1- \frac{\varrho(p)}{p} \Big) \Big( 1- \frac 1p \Big)^{-1}. 
$$ 
Thus, with ${\mathfrak S}= \prod_{p\nmid H} (1-{\varrho(p)}/p)(1-1/p)^{-1} >0$ the positive term in $S$ exceeds 
$$ 
\frac{x}{4\phi(H)} \frac{{\mathfrak S}}{\log z}. 
$$


It remains to estimate the contributions from negative terms to (\ref{4.1}), which we split into three parts depending on the 
size of $p$.   Divide the primes $z \le p\le f(2x)$ into the three ranges $z \le p \le x/z^9$, $x/z^9 \le p \le xz^9$ and $xz^9 <p \le f(2x)$.  
Corresponding to these ranges, define 
$$
 S_1 = \sum_{\substack{x\le  n\le 2 x \\ n\equiv \nu \bmod H \\ (f(n),P^{\dagger}) = 1 } }
 \sum_{\substack{p\equiv 1 \bmod{h} \\ z\le p \le xz^{-9}\\ p\mid f(n)} } 1,
$$
similarly define $S_2$ and $S_3$.  
Thus 
\begin{equation}\label{4.2}
S\ge \frac{1}{4} \frac{{\mathfrak S}}{\log z}  \frac{x}{\phi(H)}-S_1 -S_2 -S_3. 
\end{equation}

The sum $S_1$ may be upper bounded as in the previous section. In the current context an upper bound sieve produces
$$ 
S_1 \ll \frac{{\mathfrak S}}{\phi(H)}  \sum_{\substack{p\equiv 1 \bmod{h} \\ z\le p \le xz^{-9}}} \frac{x}{p\log z} \ll \frac{{\mathfrak S}}{\phi(H)} 
\frac{x}{\log z} \frac{ \log((1/\delta)-9)}{\phi(h)} \le \frac{1}{20} \frac{\mathfrak S}{\phi(H)} \frac{x}{ \log z}, 
$$
provided $h$ is large enough compared to $1/\delta$.

The sum $S_3$ also accepts treatment following the pattern laid out in the preceding section.  If $p | f(n)$ and $p>xz^9$, we put
 $f(n) =p r$ so that $r \le f(2x)/p \ll xz^{-9}$ with $r\equiv f(\nu)\bmod h$.  Given such a small value of $r$, the 
problem then amounts to requiring $f(n)$ to be a multiple of $r$ (which means that $n$ lies in one of ${\varrho(r)}$ residue classes $\bmod{\ r}$), and also lying in the residue class $\nu \bmod H$.  
Therefore, once again by the sieve,  
\begin{align*}
S_3 &\le \sum_{\substack{ r\ll xz^{-9} \\ r\equiv f(\nu) \bmod h \\ (r,P^{\dagger}H)=1}} \sum_{\substack{x\le n\le 2x \\ r| f(n) \\ n\equiv \nu \bmod H \\ (n, P^{\dagger})=1}} 
1 \ll \frac{{\mathfrak S}}{\phi(H)} \frac{x}{\log z} \sum_{\substack{ r\ll xz^{-9} \\ r \equiv f(\nu) \bmod h \\ (r,P^{\dagger}H)=1}} \frac{\varrho(r)}{r} \\
&\le  \frac{{\mathfrak S}}{\phi(H)} \frac{x}{ \log z} \frac{C(\delta)}{\phi(h)} \le \frac{1}{20} \frac{{\mathfrak S}}{\phi(H)} \frac{x}{ \log z},
\end{align*} 
in which $C(\delta)$ denotes a constant depending only on $\delta$, and $h$ is assumed to be large in comparison with $1/\delta$.

\medskip

Finally we turn to the sum $S_2$.
As before, we write $n = pr$ with $p\equiv 1 \bmod h$ and $xz^{-9} \le p\le xz^{9}$ so that the complementary variable $r$ satisfies 
$r\equiv f(\nu) \bmod h$ and $x/z^{9} \ll r \ll xz^9$.   
We sum over $r$ instead of $p$ and exchange the order of summation to see that
\begin{equation}
\label{4.3} 
S_2 \ll \sum_{\substack{xz^{-10}\le r\le xz^{10} \\ r \equiv f(\nu) \bmod h \\ (r,HP^{\dagger})=1}}{\sum_{\substack{ x \le n \le 2x \\ n\equiv \nu \bmod H \\ r|f(n) \\ (f(n), P^{\dagger}) =1 }}  1} . 
\end{equation}
Anticipating an application of Poisson summation, it is convenient to smooth the sum over $n$ above.   For concreteness,  let 
$\Phi:\mathbb R \to \mathbb R$ be the smooth function defined by $\Phi(0)=1$ and for $t\ne 0$ by
$$ 
\Phi(t) = \Big(\frac{\sin t}{t}\Big)^2.  
$$
Since $\Phi$ is always non-negative, and $\Phi(t) \gg 1$ for $1\le t\le 2$, we may bound $S_2$ by $ \ll S_2^{\prime}$ where 
\begin{equation} 
\label{4.4} 
 S'_2=\sum_{\substack{ xz^{-10}\le r\le xz^{10} \\ r\equiv f(\nu) \bmod h \\ (r,H)=1}}
\sum_{\substack{n\in \mathbb Z \\ n \equiv \nu \bmod H \\ r|f(n) \\ (f(n), P^\dagger)=1 }} \Phi\Big(\frac{n}{x}\Big). 
\end{equation} 

We treat the sieving condition $(f(n),P^\dagger)=1$ by  
Selberg's upper bound sieve.  Put $\theta_1=1$ and let
$\theta_d$ be real numbers with $\theta_d =0$ unless $d\le z$ is square-free with $d| P^{\dagger}$.    
Write
 \begin{equation}\label{weights}
 \lambda_d = \sum_{[d_1,d_2] =d} \theta_{d_1} \theta_{d_2},  
\end{equation}
so that $\lambda_d$ is non-zero only for $d$ that are square-free divisors of $P^{\dagger}$ with $d\le z^2$.   
With this notation 
\begin{equation}\label{S}
S'_2 \le 
 \sum_{\substack{ xz^{-10}\le r\le xz^{10} \\ r\equiv f(\nu) \bmod h \\ (r,H)=1} } 
  \sum_{\substack{n\in\mathbb Z \\ n \equiv \nu \bmod H \\ r|f(n)}} 
 \Phi\Big(\frac{n}{x}\Big) \Big( \sum_{d| (P^{\dagger},f(n))}  \theta_d \Big)^2 =   \sum_{d|P^\dagger} \lambda_d T(d), 
\end{equation}
where
\begin{equation} 
\label{4.5} 
T(d) = \sum_{\substack{ xz^{-10} \le r \le xz^{10} \\ r\equiv f(\nu) \bmod h \\ (r,H)=1} } 
 \sum_{\substack { n\in\mathbb Z \\ n \equiv \nu \bmod H \\ [r,d] | f(n) }}  \Phi\Big(\frac{n}{x}\Big) . 
 \end{equation}

\begin{lemma}\label{lemmaU} With notations as above, uniformly for $d\le z^2$ with $d| P^{\dagger}$ 
we have 
$$ 
T(d) = \frac{x}{H}   \sum_{\substack{xz^{-10} \le r\le xz^{10} \\ r\equiv f(\nu) \bmod h \\ (r,H)=1 }} 
\frac{\varrho([d,r])}{[d,r]} + O(x^{63/64}z^{15}).    
$$ 
%
\end{lemma}

We postpone the proof of this lemma to the next section, and proceed to complete the estimation of $S_2$.
The next lemma provides an asymptotic formula for the sum over $r$ appearing in Lemma \ref{lemmaU}. 

\begin{lemma}\label{Lemma7}  Let $D_0$ be the fundamental discriminant corresponding to $D$, so that $D/D_0$ is a perfect 
square.    If $d\le z^2$ is a divisor of $P^{\dagger}$ then 
$$ 
\sum_{\substack{ xz^{-10} \le r\le xz^{10} \\ r\equiv f(\nu) \bmod h \\ (r,H)=1 }} \frac{\varrho([d,r])}{[d,r]} = (20 \log z) 
\frac{g(d)}{d} \frac{\phi(H)}{H} \prod_{p\nmid H} \Big( 1+ \frac{(\frac{D}{p})}{p}\Big) \frac{1}{\phi(h)} \Big( 1+ \delta(D_0|h)\Big) + O(x^{-\frac 18}z^{10}), 
$$
where $\delta(D_0|h)$ equals $1$ if $D_0$ divides $h$, and equals $0$ otherwise, and $g$ is a multiplicative function given by 
$$ 
g(d) = {\varrho (d)} \prod_{p|d} \frac{(2p-1)}{(p+1)}. 
$$
 \end{lemma}

Lemma 7 will be proved in the final section of this paper. Here we continue with the estimation of $S_2^{\prime}$.  Using Lemmas 6 and 7 in 
\eqref{S} we obtain 
$$ 
S_2^{\prime} \le \sum_{\substack{ d| P^{\dagger} \\ d \le z^2}}  \lambda_d T(d) = 
T(1) \sum_{\substack{ d| P^{\dagger} \\ d \le z^2}} \lambda_d \frac{g(d)}{d} + O\Big( x^{\frac {63}{64}} z^{15} \sum_{d\le z^2} |\lambda_d| \Big). 
$$ 
We follow the familiar procedure of Selberg's sieve to minimize the main term above, which is a quadratic form in the $\theta_d$, 
subject to the linear constraint $\theta_1=1$.   As is well known, the optimal $\theta_d$ satisfy $|\theta_d|\le 1$ (see \cite[(7.9)]{FI}) 
so that $\lambda_d \ll d^{\epsilon}$ and the error term above may be bounded as $O(x^{99/100})$ provided $\delta$ is small enough.     As for the main term, note that $g(p) =0$ if $(\frac{D}{p})=-1$ and $g(p) = 4+O(1/p)$ if $(\frac{D}{p}) =1$, so that the problem corresponds to a sieve of dimension $2$.  Carrying out the Selberg sieve in this context  (see Theorem 7.1 and Proposition 7.3 of \cite{FI}) we conclude that 
$$ 
S_2^{\prime} \ll T(1) \prod_{\substack{p \le z\\ p\nmid H}} \Big( 1- \frac{g(p)}{p} \Big). 
$$  
After a small calculation, it follows that 
$$
S_2 \ll S_2^{\prime} \ll \frac{{\mathfrak S}}{\phi(H)} \frac{x}{\phi(h) \log z} \ll \frac{1}{20} \frac{{\mathfrak S}}{\phi(H)} \frac{x}{\log z}, 
$$ 
provided $h$ is large enough.

Theorem 4 is now available: we take $\delta>0$ small and $h$ suitably large in terms of $1/\delta$, so that   the estimates of $S_1$, $S_2$ and $S_3$ hold.   
Then for all sufficiently large $z$ (here large may depend on $f$) one has 
$$
S_1+ S_2 +S_3 \le  \frac 3{20} \frac{{\mathfrak S}}{\phi(H)} \frac{x}{\log z}, 
$$ 
and we  we conclude from  \eqref{4.2} that $S\gg {\mathfrak S} x(\phi(H) \log z)^{-1}$, as desired.


\section{An auxiliary estimate: Proof of Lemma \ref{lemmaU}}

In the definition of $T(d)$, we group terms according to $(r,d)$ which we denote by 
$u$.  Thus 
\begin{equation} 
\label{6.1} 
T(d) = \sum_{u| d} \sum_{\substack{ xz^{-10} \le r \le xz^{10} \\ r\equiv f(\nu) \bmod h \\ (r,d)=u\\ (r,H)=1} } 
\sum_{\substack{ n\in{\Bbb Z} \\ n \equiv \nu \bmod H \\ r(d/u) | f(n) }} \Phi \Big( \frac nx\Big). 
\end{equation} 

We now focus on the inner sum over $n$ above.  Temporarily, we put $f_{\nu}(n) = f(\nu +nH)$ so 
that the inner sum over $n$ in \eqref{6.1} may be written as 
\begin{equation} 
\label{6.2} 
\sum_{\substack{ n \in {\Bbb Z} \\ r(d/u) | f_{\nu}(n)} } \Phi \Big( \frac{\nu +nH}{x}\Big) = 
\sum_{\substack{ 1\le \xi \le r(d/u) \\ f_{\nu}(\xi) \equiv 0 \bmod {rd/u}} } \sum_{\substack{ n \in {\Bbb Z} \\ n\equiv \xi \bmod {rd/u}} } \Phi \Big( \frac{\nu +nH}{x}\Big).
\end{equation} 
Here we parametrize the inner sum by $n=\xi+r(d/u) m$ and apply the Poisson summation formula to the sum over $m$.
The Fourier transform of $\Phi$ is
$$
\widehat\Phi(t) = \int_{-\infty}^\infty \Phi(\alpha) e(-\alpha t)\, \mathrm{d}\alpha = \max (0, 1-|t|),
$$ 
and we  find that
$$
 \sum_{\substack { n\in\mathbb Z \\ n\equiv \xi \bmod {rd/u}}} \Phi\Big(\frac{\nu+nH}{x}\Big)
= \frac{x}{Hr(d/u)} \sum_{m\in\mathbb Z} e\Big(\frac{m(\nu+H\xi)}{Hr(d/u)}\Big) \widehat\Phi\Big(\frac{xm}{Hr(d/u)}\Big).
$$
Inserting this into \eqref{6.2} brings in the sum
$$ 
\varrho^{(\nu)}_m(q) = \sum_{\substack{  \xi=1 \\ f_\nu(\xi)\equiv 0 \bmod q}}^{q} e\Big(\frac{m\xi}{q}\Big),
$$
which has been studied by Hooley \cite{H}, Duke, Friedlander and Iwaniec \cite{DFI} and Toth \cite{T}, and we find that 
\begin{equation} 
\label{6.3}  
T(d) = \frac{x}{H} \sum_{u|d} \frac{u}{d}  \sum_{\substack{xz^{-10} \le r \le xz^{10} \\ r\equiv f(\nu)\bmod h \\ (r,d)=u \\ (r,H)=1} }
\frac{1}{r}  \sum_{m\in\mathbb Z} e\Big(\frac{m\nu}{Hr(d/u)} \Big)  \varrho^{(\nu)}_m(rd/u)  \widehat\Phi\Big(\frac{xm}{Hr(d/u)}\Big).
\end{equation} 

Consider first the term $m=0$ in \eqref{6.3}.   Note that $(r,H)=1$ so that $(r,h)=1$, and since $d | P^{\dagger}$ we also have $(d,h)=1$.  
It follows that $\varrho_0^{\nu}(rd/u) = \varrho(rd/u) = {\varrho}([d,r])$, and so the contribution of the $m=0$ term matches the main term of Lemma \ref{lemmaU}. 

This leaves us with the terms where $m\neq 0$.   Since ${\widehat \Phi}(t) =0$ for $|t| \ge 1$, only terms with $x|m|< Hr(d/u)$ make a non-zero contribution.   For such values of $m$, note that $|m \nu/(Hrd/u)| \le |\nu|/x \le H/x$ so that $e(m\nu/(hrd/u)) = 1+ O(H/x)$.  
 Using the trivial estimate $\varrho^{(\nu)}_m(q)\ll q^\epsilon$ when considering the
contribution arising from the $O(k/x)$, we readily find that the terms with $m\neq 0$ yield
\begin{equation} 
\label{6.4} 
\frac{x}{H}\sum_{u|d} \frac ud \sum_{m\neq 0} \sum_{\substack{xz^{-10} \le r \le xz^{10} \\ r\equiv f(\nu)\bmod h \\ (r,d) =u\\ (r,H)=1}  } \frac{  \varrho^{(\nu)}_m(rd/u)}{r}  \widehat\Phi\Big(\frac{xm}{Hr(d/u)}\Big) + O(z^{15}).  
\end{equation}

To bound the sum over $r$ above, we invoke the work of Toth \cite{T}.   His formula (16) with $L=8$, provides the 
estimate
$$
\sum_{R<r\le 2R}  \varrho^{(\nu)}_m (Ar) e\Big(\frac{jr}{h}\Big) \ll R^{63/64} A^{1/32}.  
$$
By M{\" o}bius inversion we can also impose a coprimality condition on $r$ above, thus obtaining 
$$ 
\sum_{\substack{ R<r\le 2R \\ (r,B)=1}}  \varrho^{(\nu)}_m (Ar) e\Big(\frac{jr}{h}\Big) \ll R^{63/64} (AB)^{1/32}. 
$$ 
Using the orthogonality of additive characters, we may further restrict $r$ to any given progression $\bmod {\ h}$: 
$$ 
\sum_{\substack{ R<r\le 2R \\ (r,B)=1 \\ r \equiv c \bmod h}}  \varrho^{(\nu)}_m (Ar)  \ll R^{63/64} (AB)^{1/32}. 
$$ 
Using this estimate and partial summation it is easy to see that the quantity in \eqref{6.4} is 
$$ 
\ll \frac xH \sum_{u|d} \frac ud \sum_{|m| \le Hz^{12} }  (xz^{-10})^{-1/64} d^{1/32} + z^{15} \ll x^{63/64} z^{15},  
$$ 
which completes our proof.  
 
\section{Quadratic congruences on average:  Proof of Lemma \ref{Lemma7}}

If $(\frac{D}{p}) = -1$ for any prime $p|d$, then ${\varrho}(p) =0$ and so ${\varrho}([d,r]) =0$ for all $r$.  
In this case the lemma holds trivially, and henceforth we assume that $(\frac{D}{p})=1$ for all primes $p|d$.   

Let $d \le z^2$ be a square-free divisor of $P^{\dagger}$ and let $\chi$ be a Dirichlet character $\bmod{\  h}$.  Define 
$$ 
F(s;d,\chi) = \sum_{\substack{ r= 1 \\ (r, H)=1}}^{\infty} \frac{\varrho([d,r]) \chi(r)}{[d,r]r^s}. 
$$
Note that $\varrho(q)$ is a multiplicative function of $q$, and for a prime $p \nmid H$ it is easy to see that $\varrho(p^{\ell}) = 1+ (\frac{D}{p})$ 
for all $\ell \ge 1$.   Therefore the series defining $F(s;d,\chi)$ converges absolutely in the  region Re$(s)> 0$.   Further, a small calculation with 
Euler products establishes that 
\begin{equation} 
\label{7.1} 
F(s;d, \chi) = \frac{{\varrho}(d)}{d} \frac{L(s+1,\chi) L(s+1,\chi(\tfrac{D}{\cdot}))}{L(2s+2,\chi^2 (\tfrac{D}{\cdot})^2)}   F_1(s;\chi)  F_2(s;d,\chi), 
\end{equation} 
where 
\begin{equation} 
\label{7.2} 
F_1(s;\chi) =  \prod_{p| H} \Big(1-\frac{\chi(p)}{p^{s+1}}\Big)^{-1} \Big( 1+ \frac{\chi(p)(\frac{D}{p})}{p^{s+1}}\Big), 
\end{equation} 
and 
\begin{equation} 
\label{7.3} 
F_2(s;d,\chi) =   \prod_{p|d} \Big(1+\frac{\chi(p)}{p^s} - \frac{\chi(p)}{p^{s+1}}\Big) \Big(1+ \frac{\chi(p)(\frac{D}{p})}{p^{s+1}}\Big)^{-1}.
\end{equation} 

These expressions furnish a meromorphic continuation of $F(s;d,\chi)$ to the region Re$(s)> -1/2$, with simple poles at $s=0$ only in the 
cases when $\chi$ is the principal character $\bmod h$, or when $\chi (\frac{D}{\cdot})$ is principal (which can only happen if the 
fundamental discriminant dividing $D$ is also a divisor of $h$).   Further, using the convexity bound for the Dirichlet $L$-functions appearing above, in the region Re$(s) \ge -\frac 14$ (and away from the potential pole at $s=0$) we have 
\begin{equation} 
\label{7.4} 
|F(s;d,\chi)| \ll \frac{{\varrho(d)}}{d^{\frac 14}} (1+|s|)^{\frac 14}. 
\end{equation} 

With these facts in hand, we can proceed with a standard argument in analytic number theory, using a quantitative form of 
Perron's formula and shifting contours.  We begin with Perron's formula  
$$ 
\sum_{\substack{ xz^{-10} \le r\le xz^{10} \\ (r,H) =1} } \frac{\varrho([d,r])}{[d,r]} \chi(r) 
= \frac{1}{2\pi i} \int_{\text{Re}(s) =1/\log x} F(s;d,\chi) \frac{(xz^{10})^s - (xz^{-10})^s}{s} ds. 
$$ 
After truncating the integral at Im$(s) = \sqrt{x}$, and shifting contours to the line Re$(s) =-\frac 14$ and using 
\eqref{7.4}, we obtain that the above equals 
\begin{equation} 
\label{7.5} 
(20 \log z)  {\mathop{\text{Res}}_{s=0}} F(s;d,\chi) + O(\varrho(d) x^{-\frac{1}{8}} z^{10}).  
\end{equation}    

 It remains to calculate the residue of $F(s;d,\chi)$ in cases where a pole occurs (namely, when $\chi$ is principal, or when $\chi (\frac{D}{\cdot})$ is principal).  
 When $\chi$ is the principal character $\bmod h$, a small calculation gives 
 \begin{equation} 
 \label{7.6} 
  {\mathop{\text{Res}}_{s=0}} F(s;d,\chi)  = \frac{\varrho(d)}{d} \prod_{p|d} \Big( \frac{(2p-1)}{(p+1)} \Big) \frac{\phi(H)}{H} \prod_{p\nmid H} \Big( 1+ \frac{(\frac{D}{p})}{p} \Big).
  \end{equation} 
  When $\chi (\frac{D}{\cdot})$ is the principal character (which is only possible if $D_0$, the fundamental discriminant corresponding to $D$, divides $h$)  then 
  a similar calculation shows that the residue of $L(s;d,\chi)$ is exactly the same as the right side above. 
  
 We now assemble the observations made above to complete the proof of the lemma.  Using the orthogonality of characters $\bmod h$ we have 
 $$ 
 \sum_{\substack{ xz^{-10} \le r\le xz^{10} \\ (r,H) =1 \\ r\equiv f(\nu) \bmod h }}  \frac{\varrho([d,r])}{[d,r]} = 
 \frac{1}{\phi(h)} \sum_{\chi \bmod h} \overline{\chi}(f(\nu)) \sum_{\substack{ xz^{-10} \le r\le xz^{10} \\ (r,H) =1} } \frac{\varrho([d,r])}{[d,r]} \chi(r). 
 $$    
 From \eqref{7.5} and \eqref{7.6}  the above equals 
 $$ 
( 20 \log z) \frac{g(d)}{d} \frac{\phi(H)}{H} \prod_{p\nmid H} \Big( 1+ \frac{(\frac Dp)}{p}\Big) \frac{1}{\phi(h)} \Big( 1 + \delta(D_0|h) \Big(\frac{D}{f(\nu)}\Big)\Big) + O(x^{-\frac  18} z^{10}). 
$$ 
Lastly, note that since $f(\nu)$ is coprime to $D$, and $4a f(\nu)  = (2a\nu+b)^2 -D$, one has $(\frac{D}{f(\nu)})=1$.  The lemma follows. 

 \end{document}